\documentclass[12pt]{amsart}
\usepackage{amssymb,amsthm,amsmath,epsfig,latexsym,calc}

\newlength{\bracewidth}

\newcommand{\K}{\mathbb{K}}

\numberwithin{equation}{section}
\newtheorem{theorem}{Theorem}[section]

\newtheorem{proposition}[theorem]{Proposition}

\theoremstyle{definition}

\newtheorem{example}[theorem]{Example}

\begin{document}


\title{A commutative algebraic approach to the fitting problem}

\author{\c Stefan O. Toh\v aneanu}
\address{Department of Mathematics\\
The University of Western Ontario\\
London, ON N6A 5B7, Canada}
\email{stohanea@uwo.ca}
\urladdr{http://www.math.uwo.ca/$\sim$stohanea/}

\subjclass[2000]{Primary 52C35; Secondary 13P25, 13P20}
\keywords{index of nilpotency, fat points, minimum distance}

\begin{abstract}
Given a finite set of points $\Gamma$ in $\mathbb P^{k-1}$ not all contained in a hyperplane, the ``fitting problem'' asks what is the maximum number $hyp(\Gamma)$ of these points that can fit in some hyperplane and what is (are) the equation(s) of such hyperplane(s). If $\Gamma$ has the property that any $k-1$ of its points span a hyperplane, then $hyp(\Gamma)=nil(I)+k-2$, where $nil(I)$ is the index of nilpotency of an ideal constructed from the homogeneous coordinates of the points of $\Gamma$. Note that in $\mathbb P^2$ any two points span a line, and we find that the maximum number of collinear points of any given set of points $\Gamma\subset\mathbb P^2$ equals the index of nilpotency of the corresponding ideal, plus one.
\end{abstract}
\maketitle


\section{Introduction}
Let $\K$ be any field, and let $\Gamma\subseteq \mathbb P_{\mathbb K}^{k-1}$ be a finite reduced set of points, not all contained in a hyperplane. Let $hyp(\Gamma)$ be the maximum number of points of $\Gamma$ contained in some hyperplane.

Computationally, the ``fitting problem'' (or ``exact fitting problem'') asks for effective methods or algorithms to compute this number and to find the equation of the hyperplane. If $\Gamma$ is {\em $(k-2)-$generic} (i.e., any $k-1$ of the points span a hyperplane)\footnote{In \cite{ed} and \cite{gor}, the set of points $\Gamma$ is in the affine space $\mathbb A^{k-1}$. Nevertheless we can embed it into $\mathbb P^{k-1}$ by adding to each point of $\Gamma$ an extra coordinate that equals to $1$.}, by \cite{ed} (or \cite{gor}, Algorithm MinN1), the problem of finding the hyperplane can be solved in $O(|\Gamma|^{k-1})$ time. If one knows $hyp(\Gamma)$, in \cite{gor}, Corollary 3.3 is presented an algorithm that finds the points in this hyperplane in $O(\min\{\frac{|\Gamma|^{k-1}}{hyp(\Gamma)^{k-2}} \log(\frac{|\Gamma|}{hyp(\Gamma)}),|\Gamma|^{k-1}\})$ time.

The fitting problem is in direct connection with the computation of the minimum distance $d$ of the equivalence class of linear codes with generating matrix having as columns the coordinates of the points; the points can be placed as columns in any order, and homogeneous coordinates are the same up to multiplication by a nonzero constant\footnote{For background on linear codes we recommend \cite{clo}.}. The connection is that $d=|\Gamma|-hyp(\Gamma)$. With this, \cite{gls}, \cite{h}, \cite{toh1}, \cite{toh2}, \cite{toh3}, or \cite{tohVan} show that the minimum distance gives bounds for homological invariants of zero-dimensional projective schemes. We should mention also that the hyperplanes that contain $hyp(\Gamma)$ points of $\Gamma$ are in one-to-one correspondence with the (projective) codewords of minimum weight of the class of linear codes constructed from $\Gamma$: the coefficients of the linear form defining such a hyperplane are the coefficients of the linear combination of the rows of the generating matrix of the linear code that will give the codeword of minimum weight. Therefore, Lemma 2.2 in \cite{toh2} allows us to determine these hyperplanes by finding the minimal primes of a certain ideal, which is very simple once we know $hyp(\Gamma)$.

In this paper we link $hyp(\Gamma)$ to the index of nilpotency of an ideal generated by products of linear forms. We can do this if $\Gamma\subset\mathbb P^{k-1}$ is $(k-2)-$generic. This restriction does not occur if $k-1=2$, and therefore we give a new interpretation to the fitting problem of any set of points in the plane.

The article is structured as follows. In the next section we generalize a result of Schenck (\cite{sch}, Lemma 3.1) known so far to be true only for $\mathbb P^2$. This result with the results of Davis and Geramita (see \cite{dg}) was essential to show that the Orlik-Terao algebra of an arrangement of lines in $\mathbb P^2$ is the homogeneous coordinate ring associated to a nef on the blowup of $\mathbb P^2$ at the singularities of the arrangement (see \cite{sch} for more details). In the final part we discuss about the index of nilpotency of fat points and we prove our main result Theorem \ref{main}.

The goal of this paper is to have an understanding of the fitting problem from a commutative algebraic point of view. Therefore our effort is directed towards presenting this abstract approach with as many details as possible also for a nonspecialist, leaving the analysis of the efficiency of the possible algorithms that may be created from our exposition to the experts in the field.

\section{A fat point scheme constructed from hyperplane arrangements}

Let $\mathcal A$ be a central essential hyperplane arrangement in $V=\mathbb K^k$. Suppose $\mathcal A=\{H_1,\ldots,H_n\},$ and each $H_i$ is the vanishing of a linear form $L_i\in Sym(V^*)=\mathbb K[x_1,\ldots,x_k]$. ``Central'' means that all the hyperplanes of $\mathcal A$ pass through the origin, and ``essential'' means that the rank of $\mathcal A$ is $k$ (i.e., $codim(H_1\cap\cdots\cap H_n)=k$). For more details and background on hyperplane arrangements, we recommend \cite{ot}.

To any hyperplane arrangement one can associate the lattice of intersection $L(\mathcal A)$, which is a lattice built on the intersections of hyperplanes, with levels $L_j(\mathcal A)$: $X=H_{i_1}\cap\cdots\cap H_{i_s}\in L_j(\mathcal A)$ if and only if $codim(X)=j$. $X$ is called a \textit{flat of rank $j$}, and $s$, which is the number of hyperplanes that contain $X$, will be denoted $\nu(X)$. Flats of rank $k-1$ are called \textit{coatoms}.

We are interested in a special class of hyperplane arrangements. A central essential hyperplane arrangement $\mathcal A\subset\mathbb K^k$ will be called \textit{$k-2$ generic} if and only if any $k-1$ of the linear forms $L_i$ defining the hyperplanes of $\mathcal A$ are linearly independent. Observe that all arrangements of rank 3 are $k-2$ generic (unless they are multiarrangements).

For a hyperplane arrangement $\mathcal A\subset\mathbb K^k$, define $$I_j(\mathcal A)=\langle \{L_{i_1}\cdots L_{i_j}:1\leq i_1<\cdots<i_j\leq n\}\rangle\subset R:=\mathbb K[x_1,\ldots,x_k],$$ the ideal generated by all the distinct $j$ products of the linear forms defining the hyperplanes of $\mathcal A$.

With all of these we can generalize the result of Schenck to arbitrary rank. But first, one more definition. Let $I$ be a homogeneous ideal in the polynomial ring $R:=\mathbb K[x_1,\ldots,x_k]$. The \textit{saturation of $I$} is the ideal $$I^{sat}=\{f\in R|f\in I:\langle x_1,\ldots,x_k\rangle^{n(f)}\mbox{ for some }n(f)\}.$$

\begin{proposition}\label{Schenck} Let $\mathcal A$ be a central essential $k-2$ generic arrangement of $n$ hyperplanes in $\mathbb K^k$. Then

$$I_{n-k+2}(\mathcal A)=\bigcap_{X\in L_{k-1}(\mathcal A)}I(X)^{\nu(X)-k+2}.$$
\end{proposition}
\begin{proof} Following the same considerations as in Section 4 in \cite{toh2}, any minimal prime of the ideal $I_{n-k+2}(\mathcal A)$ is of the form $$\langle L_{i_1},\ldots,L_{i_{k-1}}\rangle.$$ So all the minimal primes have codimension exactly $k-1$ (since $\mathcal A$ is $k-2$ generic), and they are the ideals of the coatoms of $\mathcal A$.

From this we get: 1) the codimension of $I_{n-k+2}(\mathcal A)$ is $k-1$, and 2) $I_{n-k+2}(\mathcal A)$ might have an embedded prime, the irrelevant ideal $\langle x_1,\ldots,x_k\rangle$. So the primary decomposition of $I_{n-k+2}(\mathcal A)$ is $$I_{n-k+2}(\mathcal A)=Q_1\cap\cdots\cap Q_s\cap J,$$ where $Q_i$ are primary ideals of codimension $k-1$ and $J$ is a $\langle x_1,\ldots,x_k\rangle-$primary ideal of codimension $k$. Since for any two ideals $A,B\subset R$ we have $(A\cap B)^{sat}=A^{sat}\cap B^{sat}$, and because $J^{sat}=R$ we obtain that $$I_{n-k+2}(\mathcal A)^{sat}=Q_1\cap\cdots\cap Q_s.$$ Here we used the fact that $Q_i$ are primary ideals of codimension $k-1$ and therefore $(Q_i)^{sat}=Q_i$. If $f\in (Q_i)^{sat}$ but $f\notin Q_i$, then $\langle x_1,\ldots,x_k\rangle^{n(f)}\cdot f\subseteq Q_i$, for some positive integer $n(f)$. From the definition of the primary ideals, we obtain that $\langle x_1,\ldots,x_k\rangle^u\subseteq Q_i$, for some power $u>0$, which is in contradiction with $codim(Q_i)=k-1$.

To prove our assertion, first we show that $I_{n-k+2}(\mathcal A)^{sat}$ has the primary decomposition described in the statement, by localizations of $I_{n-k+2}(\mathcal A)$ at each of its minimal primes. And then we show that $I_{n-k+2}(\mathcal A)=I_{n-k+2}(\mathcal A)^{sat}$.

\vskip .1in

Let $X\in L_{k-1}(\mathcal A)$ be a coatom. Denote $\nu(X)=m$, and assume that $$X=H_1\cap\cdots\cap H_m,$$ with $I(X)=\langle x_1,\ldots,x_{k-1}\rangle\subset R:=\mathbb K[x_1,\ldots,x_k]$.

If we localize $R$ at $I(X)$, we have that $L_{m+1},\ldots,L_n$ are invertible elements and therefore $$I_{n-k+2}(\mathcal A)R_{I(X)}= \langle \{L_{i_1}\cdots L_{i_{m-k+2}}:1\leq i_1<\cdots<i_j\leq m\}\rangle R_{I(X)}.$$

Consider now the linear code with generating matrix $A$ given by the coefficients of the linear forms $L_1,\ldots,L_m$. These are linear forms in variables $x_1,\ldots,x_{k-1}$. So this linear code has length $m$ and dimension $k-1$. Since any $k-1$ of these linear forms are linearly independent, the maximum number of columns of $A$ that span a $k-2$ dimensional vector space is $k-2$. So, by \cite{tohVan}, Remark 2.3, the minimum distance of this code is $$d=m-(k-2)=m-k+2.$$

From \cite{toh2}, Theorem 3.1, we have that $$\langle \{L_{i_1}\cdots L_{i_{m-k+2}}:1\leq i_1<\cdots<i_j\leq m\}\rangle=\langle x_1,\ldots,x_{k-1}\rangle^{m-k+2}.$$

So we got that $$I_{n-k+2}(\mathcal A)R_{I(X)}=I(X)^{\nu(X)-k+2}R_{I(X)},$$ for all the minimal primes (which are the ideals $I(X)$ of all the coatoms $X$) of $I_{n-k+2}(\mathcal A)$. This means that the primary decomposition of $I_{n-k+2}(\mathcal A)^{sat}$ is the one desired.

\vskip .1in

Any $I_i(\mathcal A)$ can be generated by the maximal minors of a certain matrix. Let $M$ be a $i\times n$ matrix with entries in $\mathbb K$ such that all the $i\times i$ minors are nonzero. There exists such a matrix since, by the way it is defined, the set of all these matrices is an open Zariski set. Then the maximal minors of the $i\times n$ matrix with entries in $R_1$ $$N=M\cdot\left[
\begin{array}{cccc}
L_1&0&\cdots &0 \\
0&L_2&\cdots&0 \\
\vdots &\vdots& &\vdots \\
0&0&\cdots&L_n
\end{array}
\right]$$ are the generators of $I_i(\mathcal A)$.

When $i=n-k+2$, and $\mathcal A$ is $k-2$ generic, the codimension of $I_i(\mathcal A)$ is exactly $(i-i+1)(n-i+1)=k-1$. So $R/I_{n-k+2}(\mathcal A)$ is a determinantal ring which are known to be Cohen-Macaulay (\cite{e0}, Theorem 18.18).

So, the projective dimension of $R/I_{n-k+2}(\mathcal A)$ is $pd(R/I_{n-k+2}(\mathcal A))=codim(I_{n-k+2}(\mathcal A))=k-1$. This means that $Ext^k(R/I_{n-k+2}(\mathcal A),R)=0$ and this gives that $I_{n-k+2}(\mathcal A)$ cannot have an associated prime of codimension $k$. Therefore, $I_{n-k+2}(\mathcal A)=I_{n-k+2}(\mathcal A)^{sat}$. \end{proof}

For rank 3 arrangements, \cite{sch}, Lemma 3.2 presents the graded minimal free resolution for $I_{n-1}(\mathcal A)$. More generally, from the map $R^n\longrightarrow R^i$ with matrix $N$ seen in the proof above, we get a complex of $R-$modules, known as the Eagon-Northcott complex (\cite{e}, Chapter A2H, covers in full details this complex). If $depth(I_i(\mathcal A),R)=n-i+1$, then this complex is exact and therefore it provides a free resolution for $R/I_i(\mathcal A)$.

In our instance, $i=n-k+2$ and since $R/I_{n-k+2}(\mathcal A)$ is Cohen-Macaulay, $k-1$ is the smallest integer $r$ such that $Ext^r(R/I_{n-k+2}(\mathcal A),R)\neq 0$. By \cite{e0}, Proposition 18.4, $$depth(I_{n-k+2}(\mathcal A),R)=k-1=n-(n-k+2)+1,$$ and we have exactness of the Eagon-Northcott complex for the $R-$ module $R/I_{n-k+2}(\mathcal A)$.

\section{The fitting problem and the index of nilpotency}

Let $I$ be an ideal in $R=\mathbb K[x_1,\ldots,x_k]$. The \textit{index of nilpotency} of $I$, denoted $nil(I)$, is the smallest integer $s$ such that $$(\sqrt{I})^s\subseteq I.$$ For a nice exposition about this invariant we recommend \cite{vas}, Chapter 9.2.

The next result determines the index of nilpotency of the ideal of a fat point scheme in $\mathbb P^{k-1}$.

\begin{proposition} \label{index} Let $Z=m_1P_1+\cdots+m_nP_n, m_i\geq 1$ be a fat point scheme in $\mathbb P_{\mathbb K}^{k-1}$. If $I_Z\subset R=\mathbb K[x_1,\ldots,x_k]$ is the ideal of $Z$, then $$nil(I_Z)=\max\{m_1,\ldots,m_n\}.$$
\end{proposition}
\begin{proof} Let $X=\{P_1,\ldots,P_n\}\subset \mathbb P^{k-1}$ be the support of $Z$. Let $I_X\subset R$ be the ideal of $X$. Then $$\sqrt{I_Z}=I_X.$$

Denote with $s=nil(I_Z)$, with $m=\max\{m_1,\ldots,m_n\}$, and with $I_{P_i}$ the ideal of the point $P_i$. Suppose that $m=m_1$ and that $P_1=[0,\ldots,0,1]$.

We have that $$I_Z=I_{P_1}^{m_1}\cap\cdots\cap I_{P_n}^{m_n}$$ and $$I_X=I_{P_1}\cap\cdots\cap I_{P_n}.$$

Obviously, $I_X^m\subseteq I_Z$. Therefore $$m\geq s.$$

\vskip .1in

Suppose $m\geq 2$. Otherwise, $I_Z=I_X$ and therefore $nil(I_Z)=1=m$.

Also, assume that $s\leq m-1$, and let $f\in I_X$ such that $f\notin I_{P_1}^2$. There must exist such an element, otherwise $I_{P_1}^2$ would become a primary component of $I_X$ which contradicts that $I_X$ is a radical ideal.

Since $I_X^s\subseteq I_Z$, then $$f^{m-1}\in I_{P_1}^m.$$ Because $f\in I_{P_1}-I_{P_1}^2$, and since $I_{P_1}=\langle x_1,\ldots,x_{k-1}\rangle$ we have that $$f=x_1g_1+x_2g_2+\cdots+x_{k-1}g_{k-1},$$ with at least one of the polynomials $g_i$ not in $I_{P_1}$.

If $\deg(f)=d+1$, we can assume that $$f=\ell x_k^d+g,$$ where $\ell$ is a linear form in variables $x_1,\ldots,x_{k-1}$, and $g\in I_{P_1}^2$.

Then $$f^{m-1}=\ell^{m-1}x_k^{d(m-1)}+h,$$ where $h=\sum_{b=1}^{m-1}{{m-1}\choose{b}}\ell^{m-1-b}g^bx_k^{d(m-1-b)}$. By the way we constructed $\ell$ and $g$, $$\ell^{m-1-b}g^b\in I_{P_1}^{m-1+b},$$ and since $b\geq 1$, we have that the polynomial $h\in I_{P_1}^m.$

We obtain $\ell^{m-1}x_k^{d(m-1)}\in I_{P_1}^m$, which is a contradiction: the leading monomial of $\ell^{m-1}x_k^{d(m-1)}$ under any monomial order $>$ with $x_1>\cdots>x_{k-1}>x_k$ is $x_i^{m-1}x_k^{d(m-1)}$ for some $1\leq i\leq k-1$, and should belong to the (monomial) ideal $\langle x_1,\ldots,x_{k-1}\rangle^m$.

So $s\geq m$ and therefore $s=m$. \end{proof}

Now we can put together the two propositions to obtain the main result of the notes. First, to a finite set of $n$ points $\Gamma\subset\mathbb P^{k-1}$, not all on a hyperplane, we can associate the central essential (dual) arrangement of $n$ hyperplanes $\mathcal A_{\Gamma}\subset \mathbb K^k$ defined by the vanishing of the linear forms with coefficients the coordinates of the points of $\Gamma$.

\begin{theorem}\label{main} Let $\Gamma\subset \mathbb P_{\mathbb K}^{k-1}$ be a finite $(k-2)-$generic set of $n$ points, not all contained in a hyperplane. Then $$hyp(\Gamma)=nil(I_{n-k+2}(\mathcal A_{\Gamma}))+k-2.$$
\end{theorem}
\begin{proof} If $hyp(\Gamma)$ number of points lie on a hyperplane of equation $a_1x_1+\cdots+a_kx_k=0$, then, dually, the corresponding hyperplanes in $\mathcal A_{\Gamma}$ will intersect at the coatom $[a_1,\ldots,a_k]$. So $$hyp(\Gamma)=\max\{\nu(X): X\in L_{k-1}(\mathcal A_{\Gamma})\}.$$

Immediate application of Proposition \ref{index} to Proposition \ref{Schenck} gives the result. \end{proof}

\begin{example}\label{example} We end with a simple example. Let $P_1=(1,0), P_2=(1,1), P_3=(3,-1), P_4=(-3,2)$ be four points in the real plane. Find the maximum number of collinear points, and the equation(s) of the line(s) where they are positioned.

First we projectivize the problem by embedding the affine real plane into $\mathbb P^2$. So we add the extra coordinate $z=1$ to all the points to get $$\Gamma = \{Q_1=[1,0,1],Q_2=[1,1,1],Q_3=[3,-1,1],Q_4=[-3,2,1]\}\subset\mathbb P^2.$$

To find $hyp(\Gamma)$, we create $\mathcal A_{\Gamma}$ defined by the linear forms $$L_1=x+z, L_2=x+y+z, L_3=3x-y+z, L_4=-3x+2y+z,$$ and build $$I_3(\mathcal A_{\Gamma})=\langle L_1L_2L_3,L_1L_2L_4,L_1L_3L_4,L_2L_3L_4\rangle.$$ Also consider $$J=\sqrt{I_3(\mathcal A_{\Gamma})}.$$

With Macaulay 2 (\cite{Mt}), observe that $$I_3(\mathcal A_{\Gamma}):J=\langle y+2z,x+z\rangle\mbox{ and }I_3(\mathcal A_{\Gamma}):J^2=R.$$ From Theorem \ref{main}, this means that $hyp(\Gamma)=3$.

Let $I=I_1^{m_1}\cap\cdots \cap I_s^{m_s}$ be the ideal of a fat point scheme. Denote with $J=\sqrt{I}$ and suppose that $m_1=\cdots=m_p=m$ is the maximum multiplicity of any primary component of $I$. Then $$I:J^{m-1}=I_1\cap\cdots\cap I_p.$$ In the fitting problem setup, the points with ideals $I_1,\ldots,I_p$ correspond (dually) to the hyperplanes containing $hyp(\Gamma)$ number of points of $\Gamma$.

For our example if we do this operation we obtain $\langle y+2z,x+z\rangle$, which is the ideal of the point $[-1,-2,1]$. Dually, we obtained the projective line $$-x-2y+z=0,$$ which after dehomogeneization gives the line in the plane of equation $$x+2y=1.$$ Observe that the points $P_1,P_3$ and $P_4$ are collinear sitting on this line.
\end{example}

\vskip .1in

\noindent\textbf{Acknowledgements:} I wish to thank Graham Denham and Hal Schenck for useful discussions. I am very grateful to the anonymous referee for the important terminology correction and for improving and simplifying the proof of Proposition \ref{index} that allows to drop the restriction on the characteristic of the field considered initially in this result as well as in Theorem \ref{main}.


\bibliographystyle{amsalpha}

\end{document}